\documentclass[12pt,a4paper,oneside]{amsart} 

\usepackage{cite}
\usepackage[english]{babel}
\usepackage{geometry}
\usepackage{amsmath}
\usepackage{amssymb}
\usepackage{enumitem}

\setlist[enumerate]{font=\textup,itemsep=.5em,topsep=.5em}

\title{Polynomial functions on upper triangular matrix algebras }
\author[S.~Frisch]{Sophie Frisch}
\thanks{The research leading to this publication was supported by
the Austrian Science Fund FWF, grant P27816-N26.}
\address{Institut f\"ur Mathematik A, Technische Universit\"at Graz,
Kopernikusgasse 24, 8010 Graz, Austria}
\email{frisch@math.tugraz.at}

\newtheorem{thm}{Theorem}[section]
\newtheorem{prop}[thm]{Proposition}
\newtheorem{lem}[thm]{Lemma}
\newtheorem{cor}[thm]{Corollary}
\theoremstyle{definition}
\newtheorem{Def}[thm]{Definition}
\newtheorem{notn}[thm]{Notation}
\newtheorem{exa}[thm]{Example}

\newtheorem{rem}[thm]{Remark}

\newcommand{\natn}{\mathbb{N}}

\long\def\comment#1{\relax}
\long\def\commented_out#1{\relax}

\newcommand{\TnD}{{\rm T}_n(D)}
\newcommand{\TnK}{{\rm T}_n(K)}
\newcommand{\TnR}{{\rm T}_n(R)}
\newcommand{\TnS}{{\rm T}_n(S)}
\newcommand{\TnI}{{\rm T}_n(I)}

\DeclareMathOperator{\Int}{Int}  
\DeclareMathOperator{\Intl}{Int^{\ell}} 
\DeclareMathOperator{\IntK}{Int_K}  
\DeclareMathOperator{\IntR}{Int_R}  

\newcommand{\TMint}[1]{\IntR({\rm T}_{#1}(S), {\rm T}_{#1}(I))}
\newcommand{\NR}[1]{\mathrm{N}_{R}(T_{#1}(R))}
\newcommand{\TMMintn}{\Int_{\TnR}(\TnS,\TnI)}
\newcommand{\TMMintnl}{\Intl_{\TnR}(\TnS,\TnI)}
\newcommand{\TMintD}[1]{\IntK({\rm T}_{#1}(D))}
\newcommand{\TMMintnD}{\Int_{\TnK}(\TnD)}
\newcommand{\TMMintnDl}{\Intl_{\TnK}(\TnD)}
\newcommand{\TnSres}[2]{T_n^{[{#1},{#2}]}(S)} 

\newcommand{\fCl}{{f(C)_{\ell}}}

\let\tensor=\otimes 


\begin{document}
\vspace*{-1.5cm}
\hbox{\ \hfill \ }
\vspace*{-1cm}
\hbox{to appear in Monatsh.~Math. (2017)\hfill}
\vspace*{1cm}
\begin{abstract}
There are two kinds of polynomial functions on matrix algebras over
commutative rings:
those induced by polynomials with coefficients in the algebra itself
and those induced by polynomials with scalar coefficients.
In the case of algebras of upper triangular matrices over a commutative
ring, we characterize the former in terms of the latter (which are 
easier to handle because of substitution homomorphism).
We conclude that the set of integer-valued polynomials with matrix
coefficients on an algebra of upper triangular matrices is a ring, and
that the set of null-polynomials with matrix coefficients on an
algebra of upper triangular matrices is an ideal.
\end{abstract}

\maketitle

\section{Introduction}
Polynomial functions on non-commutative algebras over commutative rings
come in two flavors:
one, induced by polynomials with scalar coefficients,
the other, by polynomials with coefficients in the 
non-commutative algebra itself 
\cite{LoWe12GRiv,Wer12IVM,EvFaJo13ivlt,Fri13IVA,Per14Ivmdd,Wer14PK,
PerWer16PIP,PeWe--ntC,PeWe--DPA}.

For the algebra of upper triangular matrices over a commutative ring,
we show how polynomial functions with matrix coefficients can be 
described in terms of polynomial functions with scalar coefficients.
In particular, we express integer-vaued polynomials with matrix
coefficients in terms of integer-valued polynomials with scalar
coefficients. The latter have been studied extensively by Evrard, 
Fares, Johnson \cite{EvFaJo13ivlt} and Peruginelli \cite{Per14Ivmdd} 
and have been characterized as polynomials that are integer-valued 
together with a certain kind of divided differences.

Also, our results have a bearing on several open questions in the 
subject of polynomial functions on non-commutative algebras.
They allow us to answer in the affirmative, in the case of algebras 
of upper triangular matrices, two questions (see section \ref{applsect})
posed by N.~Werner \cite{Wer14PK} :
whether the set of null-polynomials on a non-commutative algebra
forms a two-sided ideal and
whether the set of integer-valued polynomials on a non-commutative 
algebra forms a ring.
In the absence 
of a substitution homomorphism for polynomial functions on 
non-commutative rings, neither of these properties is a given.

Also, our results on polynomials on upper triangular matrices show 
that we may be able to describe polynomial functions induced by
polynomials with matrix coefficients by polynomial functions
induced by polynomials with scalar coefficients, even when the
relationship is not as simple as in the case of the full matrix
algebra.

Let $D$ be a domain with quotient field $K$, $A$ a finitely generated
torsion-free $D$-algebra and $B=A\tensor_D K$. To exclude pathological
cases we stipulate that $A\cap K=D$. Then the set of right integer-valued
polynomials on $A$ is defined as
\[\Int_B(A) = \{f\in B[x]\mid \forall a\in A\> f(a)\in A\},\]
where $f(a)$ is defined by substitution on the right of the
coefficients, $f(a)=\sum_k b_k a^k$, for $a\in A$ and
$f(x)=\sum_k b_k x^k\in B[x]$. Left integer-valued polynomials
are defined analogously, using left-substitution:
\[\Int^{\ell}_B(A) = \{f\in B[x]\mid \forall a\in A\> 
f(a)_{\ell}\in A\},\]
where $f(a)_{\ell}=\sum_k a^k b_k$. Also, we have 
integer-valued polynomials on $A$ with scalar coefficients:
\[
\IntK(A)=\{f\in K[x]\mid \forall a\in A\> f(a)\in A\}.
\]

For $A=M_n(D)$ and $B=M_n(K)$ we could show \cite{Fri13IVA} that 
\[\Int_B(A) = \IntK(A)\tensor_D A. \]
Peruginelli and Werner \cite{PeWe--DPA} have
characterized the algebras for which this relationship holds.
We now take this investigation one step further and
show for algebras of upper triangular matrices, where $A=\TnD$ and
$B=\TnK$, that $\Int_B(A)$ can be described quite nicely in terms of
$\IntK(A)$ (cf.~Remark~\ref{uppertriangular},
Theorem~\ref{matrix-scalar-connection}), but that the connection 
is not as simple as merely tensoring with $A$.

\comment{
Describing integer-valued polynomials with matrix coefficients on 
algebras of upper triangular matrices by integer-valued polynomials 
with scalar coefficients is useful and informative, because the latter
polynomials have been studied extensively: they have been characterized
as polynomials that are integer-valued together with a certain 
kind of divided differences \cite{EvFaJo13ivlt,Per14Ivmdd}. 
}

For a ring $R$, and $n\ge 1$, let $M_n(R)$ denote the ring of
$n\times n$ matrices with entries in $R$ and $T_n(R)$ the subring of
upper triangular matrices, i.e.,
the ring consisting of $n\times n$ matrices
$C=(c_{ij})$ with $c_{ij}\in R$ and
\[c_{ij}\ne 0\Longrightarrow j\ge i.\]

\begin{rem}\label{isomorphism}
Let $R$ be a commutative ring.
We will make extensive use of the ring isomorphism between
the ring of polynomials in one variable with coefficients in $M_n(R)$
and 
the ring of $n\times n$ matrices with entries in $R[x]$:
\[
\varphi\colon (M_n(R))[x]\rightarrow M_n(R[x]),\quad
\sum_k (a_{ij}^{(k)})_{\scriptscriptstyle1\le i,j\le n}\, x^k \mapsto 
\left( \sum_k a_{ij}^{(k)} x^k \right)_{\scriptscriptstyle\!\!1\le i,j\le n}
\]
The restriction of $\varphi$ to $(T_n(R))[x]$ gives a ring 
isomorphism between $(T_n(R))[x]$ and $T_n(R[x])$.
\end{rem}

\begin{notn}\label{matrixcoeffnotation}
For $f\in (M_n(R))[x]$,
$F_k$ denotes the $k$-th coefficient of $f$, 
and $f_{ij}^{(k)}\in R$ the $(i,j)$-th entry in $F_k$. When we 
reinterpret $f$ as an element of $M_n(R[x])$ via the ring isomorphism of 
Remark~\ref{isomorphism}, we denote the $(i,j)$-th entry of $f$
by $f_{ij}$. 

In other words,
\[
f=\sum_k F_k x^k,
\quad\textrm{where}\quad 
F_k=(f_{ij}^{(k)})_{1\le i,j\le n}
\]
and
\[
\varphi(f)= (f_{ij})_{1\le i,j\le n},
\quad\textrm{where}\quad 
f_{ij}=\sum_k f_{ij}^{(k)} x^k
\]

Also, we write $\left[ M\right]_{i\,j}$ for the $(i,j)$-the entry of
a matrix $M$. In particular, $\left[f(C)\right]_{i\,j}$ is the
$(i,j)$-the entry of $f(C)$, the result of substituting the matrix $C$ 
for the variable (to the right of the coefficients) in $f$, and
$\left[\fCl\right]_{i\,j}$ the $(i,j)$-th entry of $\fCl$, the result
of substituting $C$ for the variable in $f$ to the left of the
coefficients.

\end{notn}

We will work in a setting that allows us to consider 
integer-valued polynomials and null-polynomials on
upper triangular matrices at the same time.

From now on, $R$ is a commutative ring, $S$ a subring of $R$,
and $I$ an ideal of $S$.

\begin{notn} Let $R$ be commutative ring, $S$ a subring of $R$
and $I$ an ideal of $S$. 

\begin{align*}
\TMint{n}&=\{f\in R[x]\mid \forall C\in T_n(S):\, f(C)\in T_n(I)\}\\
\TMMintn&=\{f\in (T_n(R))[x] \mid 
\forall C\in T_n(S):\, f(C)\in T_n(I)\}\\
\TMMintnl&=\{f\in (T_n(R))[x] \mid 
\forall C\in T_n(S):\, \fCl\in T_n(I)\},
\end{align*}
where 
\[
f(C)=\sum_k F_k C^k
\qquad {\textrm and}\qquad
\fCl=\sum_k  C^k F_k
\qquad {\textrm for}\qquad 
f=\sum_k F_k x^k.
\]
\end{notn}

\begin{exa}\label{ivpoly}
When $D$ is a domain with quotient field $K$ and we set $R=K$ and $S=I=D$,
then $\TMint{n}=\TMintD{n}$ is the ring of integer-valued polynomials
on $T_n(D)$ with coefficients in $K$ and
$\TMMintn=\TMMintnD$, the set of right integer-valued polynomials on
$T_n(D)$ with coefficients in $(T_n(K))$. We will show that the latter
set is closed under multiplication and, therefore, a ring, 
cf.~Theorem~\ref{ring}.
\end{exa}

\begin{exa}\label{nullpoly}
When $R$ is a commutative ring and we set $S=R$ and $I=(0)$, then
$\TMint{n}=\NR{n}$ is the ring of those polynomials in $R[x]$ that
map every matrix in $T_n(R)$ to zero, and
$\TMMintn=\mathrm{N}_{T_n(R)}(T_{n}(R))$ is the
set of polynomials in $(T_n(R))[x]$ that map every matrix in
$T_n(R)$ to zero under right substitution. 
We will show that the latter set is actually
an ideal of $(T_n(R))[x]$, cf.~Theorem~\ref{ideal}. 
\end{exa}

We illustrate here in matrix form our main result on the connection
between the two kinds of polynomial functions on upper triangular 
matrices, those induced by polynomials with matrix coefficients on 
one hand, and those induced by polynomials with scalar coefficients 
on the other hand. (For details, see Theorem
\ref{matrix-scalar-connection}.)

\begin{rem}\label{uppertriangular}
If we identify $\TMMintn$ and $\TMMintnl$ --- a priori subsets of
$(\TnR)[x]$ --- with their respective images
in $T_n(R[x])$ under the ring isomorphism
\[
\varphi\colon (\TnR)[x]\rightarrow T_n(R[x]),\quad
\sum_k (f_{ij}^{(k)})_{\scriptscriptstyle1\le i,j\le n}\, x^k \mapsto 
\left( \sum_k f_{ij}^{(k)} x^k \right)_{\scriptscriptstyle\!\!1\le i,j\le n}
\]
then

\begin{enumerate}
\item
$\TMMintn=$ 
\[
\begin{pmatrix}
\TMint{n}&\TMint{n-1}&\ldots&\TMint{2}&\IntR(S,I)\cr
0&\TMint{n-1}&\ldots&\TMint{2}&\IntR(S,I)\cr
&&\ddots&&\cr
0&0&\ldots&\TMint{2}&\IntR(S,I)\cr
0&0&\ldots&0&\IntR(S,I)\cr
\end{pmatrix}
\]
\vspace*{1ex}
\item
$\TMMintnl=$ 
\[
\begin{pmatrix}
\IntR(S,I)&\IntR(S,I)&\ldots&\IntR(S,I)&\IntR(S,I)\cr
0&\TMint{2}&\ldots&\TMint{2}&\TMint{2}\cr
&&\ddots&&\cr
0&0&\ldots&\TMint{n-1}&\TMint{n-1}\cr
0&0&\ldots&0&\TMint{n}\cr
\end{pmatrix}
\]
\comment{%
\item
$\TMMintn\cap\TMMintnl$, the ring of those polynomials with coefficients
in $\TnR$ that map matrices in $\TnS$ to matrices $\TnI$ whether
the arguments are substituted for the variable in to the right of the 
coefficients or to the left of the coefficients, equals the ring of
matrices whose $(i,j)$-th entry is in $\TMint{m}$ with 
$m=\min(i, n-j+1)$.
\[
\begin{pmatrix}
\TMint{n}&\TMint{n-1}&\ldots&\TMint{2}&\Int_R(S,I)\cr
0&\TMint{n-1}&\ldots&\TMint{2}&\TMint{2}\cr
&&\ddots&&\cr
0&0&\ldots&\TMint{n-1}&\TMint{n-1}\cr
0&0&\ldots&0&\TMint{n}\cr
\end{pmatrix}
\]
} 
\end{enumerate}
\end{rem}

\section{path polynomials and polynomials with scalar coefficients}%
\label{pathpolysect}

We will use the combinatorial interpretation of the $(i,j)$-th entry in 
the $k$-th power of a matrix as a weighted sum of paths from $i$ to $j$.

Consider a set $V$ together with a subset $E$ of $V\times V$. 
We may regard the pair $(V,E)$ as a set with a binary relation
or as a directed graph. Choosing the interpretation as a graph,
we may associate monomials to paths and polynomials to finite
sets of paths.

For our purposes, a path of length $k\ge 1$ from $a$ to $b$ (where $a,b\in V$)
in a directed graph $(V,E)$ is a sequence $e_1e_2\ldots e_{k}$ of edges
$e_i=(a_i,b_i)\in E$ such that $a_1=a$, $b_{k}=b$ and $b_{j}=a_{j+1}$
for $1\le j<k$. Also, for each $a\in V$, there is a path of length $0$
from $a$ to $a$, which we denote by $\varepsilon_a$. (For $a\ne b$
there is no path of length $0$ from $a$ to $b$.)

We introduce a set of independent variables 
$X=\{x_{v\, w}\mid (v,\, w)\in V\times V\}$ and consider the polynomial 
ring 
\[R[X]=R[\{x_{v\, w}\mid (v,\, w)\in V\times V\}]
\]
with coefficients in the commutative ring $R$. 

To each edge $e=(a, b)$ in $E$, we associate the variable $x_{ab}$
and to each path $e_1e_2\ldots e_{k}$ of length $k$, with $e_i=(a_i,b_i)$,
the monomial of degree $k$ which is the product of the variables associated 
to the edges of the path:
$x_{a_1b_1}x_{a_2b_2}\ldots x_{a_kb_k}$.
To each path of length $0$ we associate the monomial $1$.

If $E$ is finite, or, more generally, if for any pair of vertices 
$a,b\in V$ and fixed $k\ge 0$ there are only finitely many paths in 
$(V,E)$ from $a$ to $b$, we define the $k$-th path polynomial 
from $a$ to $b$, denoted by $p_{a b}^{(k)}$, as the sum in $R[X]$ of the 
monomials 
corresponding to paths of length $k$ from $a$ to $b$. If there is no
path of length $k$ from $a$ to $b$ in $(V,E)$, we set $p_{a b}^{(k)}=0$.

From now on, we fix the relation $(\natn, \le)$ and all path polynomials
will refer to the graph of $(\natn, \le)$, where 
$\natn=\{1, 2, 3, \ldots \}$, or one of its finite 
subgraphs given by intervals. The finite subgraph given by the interval
\[[i,j]=\{k\in\natn\mid i\le k\le j\}\] is the graph with 
set of vertices $[i,j]$ and set of edges $\{(a,b)\mid i\le a\le b\le j\}$.

Because of the transitivity of the relation ``$\le$'', a path in
$(\natn,\le)$ from $a$ to $b$ involves only vertices in the 
interval $[a,\, b]$. The path polynomial $p_{a b}^{(k)}$,
therefore, is the same whether we consider $a, b$ as vertices in the 
graph $(\natn,\le)$, or any subgraph given by an interval $[i,j]$ with 
$a,b\in [i,j]$.
So we may suppress all references to intervals and subgraphs and 
define:

\begin{Def}\label{pathpolydef}
Let $R$ be a commutative ring.
The $k$-th path polynomial from $i$ to $j$
(corresponding to the
relation $(\natn,\le)$) in $R[x]$ is defined by
\begin{enumerate}
\item
For $1\le i\le j$ and $k> 0$ by\\
\[p_{ij}^{(k)}=\sum_{i=i_1\le i_2\le\ldots \le i_{k+1}=j} 
x_{i_1 i_2}x_{i_2 i_3}\ldots x_{i_{k-1}i_k}x_{i_k i_{k+1}},
\]
\item
for $1\le i \le j$ and $k=0$ by $p_{ij}^{(0)}=\delta_{ij}$,
\item
for $i>j$ and all $k$: $p_{ij}^{(k)}=0$.
\end{enumerate}

For $a,b\in\natn$, we define the sequence of path polynomials from 
$a$ to $b$ as 
\[
p_{a b}=(p_{a b}^{(k)})_{k\ge 0}.
\]
\end{Def}

\begin{rem}\label{pathpolyrem}
Note that $p_{ij}^{(k)}$ is the $(i,j)$-th entry of the $k$-th
power of a generic upper triangular $n\times n$ matrix (with $n\ge i,j$)
whose $(i,j)$-th entry 
is $x_{ij}$ when $i\le j$ and zero otherwise.
\end{rem}

\begin{exa}
The sequence of path polynomials from $2$ to $4$ is
\[
p_{2 4}=(0,\;\; x_{24},\;\; x_{22}x_{24}+x_{23}x_{34}+x_{24}x_{44},\;\ldots\ )
\]
and
\[
p^{(3)}_{24}=
x_{22}x_{22}x_{24} + x_{22}x_{23}x_{34} + x_{22}x_{24}x_{44} +
x_{23}x_{33}x_{34} + x_{23}x_{34}x_{44} + x_{24}x_{44}x_{44} 
\]

Again, note that $p_{2 4}$ is the sequence of entries in position $(2,4)$ in
the powers $G^0,G, G^2, G^3,\ldots$ of a generic $n\times n$ (with $n\ge 4$)
upper triangular
matrix $G=(g_{ij})$ with $g_{ij}=x_{ij}$ for $i\le j$ and $g_{ij}=0$
otherwise.
\end{exa}

In addition to right and left substitution of a matrix for the variable 
in a polynomial in $R[x]$ or $(\TnR)[x]$, we are going to use another
way of plugging matrices into polynomials, namely, into polynomials in
$R[X]=R[\{x_{ij}\mid i,j\in\natn\}]$. For this purpose, the matrix
$C=(c_{ij})\in M_n(R)$ is regarded as a vector of elements of $R$ 
indexed by $\natn\times\natn$, with $c_{ij}=0$ for $i>n$ or $j>n$:

\begin{Def}\label{matrixsubstitution}
For a polynomial $p\in R[X]=R[\{x_{ij}\mid i, j\in\natn\}]$
and a matrix $C=(c_{ij})\in M_n(R)$ we define $p(C)$ as the result of 
substituting $c_{ij}$ for those $x_{ij}$ in $p$ with $i,j\le n$ and 
substituting $0$ for all $x_{kh}$ with $k>n$ or $h>n$.
\end{Def}

To be able to describe the $(i,j)$-th entry in $f(C)$, where
$f\in R[x]$, we need one more construction:
For sequences of polynomials $p=(p_i)_{i\ge 0}, q=(q_i)_{i\ge 0}$ 
in $R[X]$, at least one of which is finite, we define a scalar product
$\langle p, q\rangle=\sum_i p_i q_i$. Actually, we only need one
special instance of this construction, that where one of the sequences
is the sequence of coefficients of a polynomial in $R[x]$ and the
other a sequence of path polynomials from $a$ to $b$.

\begin{Def}\label{scalarproddef}
Given $f=f_1+f_1x+\ldots f_mx^m\in R[x]$ (which we identify with the 
sequence of its coefficients), $a,\, b\in\natn$, and 
$p_{a b}=(p_{a b}^{(k)})_{k=0}^\infty$ the sequence of path polynomials
from $a$ to $b$ as in Definition~\ref{pathpolydef}, we define
\[
\langle f,\, p_{a b}\rangle=\sum_{k\ge 0} f_k\, p_{a b}^{(k)}.
\]
\end{Def}

\begin{Def}\label{Sstarimage}
For a polynomial $p\in R[X]=R[\{x_{ij}\mid i,j\in\natn\}]$
and $S\subseteq R$, we define the image $p(S^{\ast})\subseteq R$ as the 
set of values of $p$ as the variables occurring in $p$ range through $S$
independently. (The star in $S^{\ast}$ serves to remind us that the
arguments of $p$ are not elements of $S$, but $k$-tuples of elements of
$S$ for unspecified $k$.)

We define $\Int(S^{\ast},I)$ as the set of those polynomials in
$R[X]=R[\{x_{ij}\mid i,j\in\natn\}]$ that take values in $I$ whenever 
elements of $S$ are substituted for the variables. 
\end{Def}

The notation $\Int(S^{\ast},I)$ is suggested by the convention that 
$\Int(S,I)$ consists of polynomials in one indeterminate mapping 
elements of $S$ to elements of $I$ and, for $k\in\natn$,
$\Int(S^{k},I)$ consists of polynomials in $k$ indeterminates mapping
$k$-tuples of elements of $S$ to $I$.

We summarize here the connection between path poynomials and the
related constructions of Definitions~\ref{pathpolydef}, 
\ref{matrixsubstitution}, \ref{scalarproddef} and \ref{Sstarimage}
 with entries of powers of matrices and
entries of the image of a matrix under a polynomial function.

\begin{rem}\label{pathpolyproperties}
Let $R$ be a commutative ring, $C\in\TnR$, $k\ge 0$, $1\le i,j\le n$, 
and $p_{ij}^{(k)}$ the $k$-the path polynomial from $i$ to $j$ 
in $R[x]$ as in Definition~\ref{pathpolydef}.
\begin{enumerate}
\item\label{pathpolypower}
$\left[C^k\right]_{i j} = p_{ij}^{(k)}(C)$
\item\label{fCij}
For $f\in R[x]$,
$\left[f(C)\right]_{i j} = \langle f, p_{i j}\rangle(C)$.
\item\label{pathpolyzero}
If the $i$-th row or the $j$-th column of $C$ is zero then
$p_{ij}^{(k)}(C)=0$, and for all $f\in R[x]$, 
$\langle f, p_{i j}\rangle(C)=0$.
\item\label{pathpolyimage}
$
p_{ij}^{(k)}(S^{\ast})=
\{p_{ij}^{(k)}(C)\mid C\in\TnS\}.
$
\end{enumerate}
\end{rem}

\begin{proof}
(\ref{pathpolypower}) and (\ref{fCij}) follow immediately from
Definitions \ref{pathpolydef}, \ref{matrixsubstitution} and
\ref{scalarproddef}.
 Compare Remark~\ref{pathpolyrem}.
(\ref{pathpolyzero}) follows from (\ref{fCij}) and
Definition~\ref{matrixsubstitution}, since every monomial
occurring in $p_{ij}^{(k)}$ involves a variable $x_{im}$ for some $m$
and a variable $x_{hj}$ for some $h$. Also, (\ref{pathpolyimage}) 
follows from Definitions~\ref{pathpolydef} and \ref{matrixsubstitution}.
\end{proof}

\begin{lem}\label{image-equiv}
Let $f\in R[x]$.
The image of $\langle f, p_{ij}\rangle$ 
under substitution of elements of $S$ 
for the variables depends only on $f$ and $j-i$,
that is, for all $i\le j$ and all $m\in\natn$
\[
\langle f, p_{ij}\rangle(S^{\ast})=
\langle f, p_{i+m\> j+m}\rangle(S^{\ast})
\]
\end{lem}

\begin{proof}
The $R$-algebra isomorphism 
\[
\psi\colon R[\{x_{hk}\mid i\le h\le k\le j\}]\rightarrow
R[\{x_{hk}\mid i+m\le h\le k\le j+m\}]
\]
with $\psi(x_{hk}) = \psi(x_{h+m\, k+m})$ and $\psi(r)=r$
for all $r\in R$ maps $\langle f, p_{ij}\rangle$ to 
$\langle f, p_{i+m\> j+m}\rangle$.

Applying $\psi$ amounts to a renaming of variables;
it doesn't affect the image of the polynomial function resulting 
from substituting elements of $S$ for the variables.
\end{proof}
\goodbreak

\begin{prop}\label{scalar-coeff-iv-criterion}
Let $f\in R[x]$. The following are equivalent
\begin{enumerate}
\item\label{scalar-ivpoly}
$f\in \Int_R(T_n(S), T_n(I))$
\item\label{each-entry}
$\forall\; 1\le i\le j\le n\quad 
\langle f, p_{ij}\rangle\in \Int(S^{\ast},I)$
\item\label{eqiv-pathpoly}
$\forall\; 0\le k\le n-1\quad
\exists i\in\natn\; \langle f, p_{i\> i+k}\rangle\in \Int(S^{\ast},I)$
\end{enumerate}
\end{prop}

\begin{proof}
The $(i,j)$-th entry of $f(C)$, for $C\in \TnR$, is
$\langle f, p_{ij}\rangle$(C), 
by Remark~\ref{pathpolyproperties} (\ref{fCij}).
If $C$ varies through 
$\TnS$, then all variables occurring in $\langle f, p_{ij}\rangle$
vary through $S$ independently.
This shows the equivalence of (\ref{scalar-ivpoly}) and
(\ref{each-entry}). 

By Lemma~\ref{image-equiv}, the image of 
$\langle f, p_{ij}\rangle$ as the variables range through $S$
depends only on $f$ and $j-i$. This shows the equivalence of 
(\ref{each-entry}) and (\ref{eqiv-pathpoly}).
\end{proof}

\section{Lemmata for polynomials with matrix coefficients}%
\label{matrixcoeffsect}

\begin{notn}\label{TnSres}
Given $C\in M_n(R)$ and $1\le h\le j \le n$,
let $C^{[h,j]}$ be the matrix obtained from $C$ by replacing 
all entries with row- or column-index outside the interval 
$[h,j]=\{m\in\natn \mid h\le m\le j\}$ by zeros; and
for $S\subseteq R$,  let
\[
\TnSres{h}{j}=\{C^{[h,j]}\mid C\in\TnS\}.
\]
\end{notn}

\begin{rem}\label{resfCijh}
Note that, for $f\in R[x]$ and $C\in\TnR$, 
\[\left[f(C)\right]_{ij}=
\langle f, p_{i j}\rangle(C)=
\langle f, p_{i j}\rangle(C^{[i,j]})=
\left[f(C^{[i,j]})\right]_{ij} .\]
This is so because no variables other than $x_{st}$ with
$i\le s\le t\le j$ occur in $p^{(k)}_{i j}$, for any $k$.
\end{rem}

We derive some technical, but useful, formulas for 
the $(i,j)$-th entry in $f(C)$ and $\fCl$, respectively, where 
$f\in (T_n(R))[x]$ and $C\in \TnR$.

\begin{lem}\label{summationchange}
Let $f\in (\TnR)[x]$ and $C\in\TnR$ with notation as in 
\ref{matrixcoeffnotation} and in Definitions \ref{pathpolydef},
\ref{matrixsubstitution}, and \ref{scalarproddef}.
Then, for all $1\le i\le j\le n$, we have

\begin{description}\itemsep10pt \itemindent=0em \leftmargin=0.5em
\item[(R)]\label{summationchange-right}

\[
\left[f(C)\right]_{i\,j} = 
\sum_{h\in [i, j]} \left[f_{ih}(C)\right]_{h\,j} =
\sum_{h\in [i, j]} \left[f_{ih}(C^{[h,j]})\right]_{h\,j}
\]
and also
\[
\left[f(C)\right]_{i\,j} = 
\sum_{h\in [i, j]} \left\langle f_{ih}, p_{hj}\right\rangle(C) =
\sum_{h\in [i, j]} \left\langle f_{ih}, p_{hj}\right\rangle(C^{[h,j]}) .
\]

\item[(L)]\label{summationchange-left}

\[
\left[\fCl\right]_{i\,j} = 
\sum_{h\in [i, j]}  \left[f_{hj}(C)\right]_{i\,h} =
\sum_{h\in [i, j]} \left[f_{hj}(C^{[i,h]})\right]_{i\,h}
\]
and also
\[
\left[\fCl\right]_{i\,j} = 
\sum_{h\in [i, j]} \left\langle f_{hj}, p_{ih}\right\rangle(C) =
\sum_{h\in [i, j]}  \left\langle f_{hj}, p_{ih}\right\rangle(C^{[i,h]}).
\]
\end{description}

\end{lem}

\begin{proof}
Let $d=\deg f$.
\[\left[f(C)\right]_{i\,j} =
\sum_{k=0}^d \left[F_k C^k\right]_{i\,j} =
\sum_{k=0}^d \sum_{h\in [1,n]} f_{ih}^{(k)} \left[C^k\right]_{h\,j}=
\sum_{k=0}^d \sum_{h\in [1,n]} f_{ih}^{(k)} p_{hj}^{(k)}(C).\]
Changing summation, we get 
$\sum_{h\in [1,n]} \sum_{k=0}^d f_{ih}^{(k)} p_{hj}^{(k)}(C)$.

Considering that $f_{ih}^{(k)}\ne 0$ only if $i\le h$, and that
$p_{hj}^{(k)}\ne 0$ only if $h\le j$, we can restrict summation
to $h\in [i, j]$. Consequently, $\left[f(C)\right]_{i\,j}$ equals 

\[
\sum_{h\in [i, j]} \sum_{k=0}^d f_{ih}^{(k)} p_{hj}^{(k)}(C) =
\sum_{h\in [i, j]} \left\langle f_{ih}, p_{hj}\right\rangle(C) =
\sum_{h\in [i, j]} \left[f_{ih}(C)\right]_{h\,j}.
\]

By Remark~\ref{resfCijh},
we can replace $C$ by $C^{[h,j]}$, the matrix obtained from $C$
by replacing all entries with row or column index outside the 
interval $[h,j]$ by zeros. Therefore, 

\[
\left[f(C)\right]_{i\,j} =
\sum_{h\in [i, j]} \left\langle f_{ih}, p_{hj}\right\rangle(C^{[h,j]}) =
\sum_{h\in [i, j]} \left[f_{ih}(C^{[h,j]})\right]_{h\,j}.
\]
This shows the formulas for right substitution.

Now, if we substitute $C$ for the variable of $f$ to the left 
of the coefficients,
\[
\left[\fCl\right]_{i\,j} =
\sum_{k=0}^d \left[C^k F_k\right]_{i\,j} =
\sum_{k=0}^d \sum_{h\in[1,n]} \left[C^k\right]_{i\,h} f_{hj}^{(k)}=
\sum_{k=0}^d \sum_{h\in[1,n]} p_{ih}^{(k)}(C) f_{hj}^{(k)}.
\]
Changing summation, we get
\[
\sum_{h\in [1,n]} \sum_{k=0}^d p_{ih}^{(k)}(C) f_{hj}^{(k)} =
\sum_{h\in [1,n]} \sum_{k=0}^d f_{hj}^{(k)} p_{ih}^{(k)}(C).
\]
Considering that $f_{hj}^{(k)}\ne 0$ only if $h\le j$, and that
$p_{ih}^{(k)}\ne 0$ only if $i\le h$, we can restrict summation
to $h\in [i, j]$. Consequently, $\left[\fCl\right]_{i\,j}$ equals
\[
\sum_{h\in [i, j]} \sum_{k=0}^d f_{hj}^{(k)} p_{ih}^{(k)}(C) =
\sum_{h\in [i, j]} \left\langle f_{hj}, p_{ih}\right\rangle(C) =
\sum_{h\in [i, j]} \left[f_{hj}(C)\right]_{i\,h}.
\]
By Remark~\ref{resfCijh},
we can replace $C$ by $C^{[i,h]}$, the matrix obtained from $C$
by replacing all entries in rows and columns with index outside
${[i,h]}$ by zeros.
Therefore,
\[
\left[\fCl\right]_{i\,j} =
\sum_{h\in [i, j]} \left\langle f_{hj}, p_{ih}\right\rangle(C^{[i,h]}) =
\sum_{h\in [i, j]} \left[f_{hj}(C^{[i,h]})\right]_{i\,h}.
\]
\end{proof}
\goodbreak

\begin{lem}\label{eachij}
Let $f\in (\TnR)[x]$ with notation as in
\ref{matrixcoeffnotation} and in Definitions \ref{pathpolydef},
\ref{matrixsubstitution}, and \ref{scalarproddef}. Let $1\le i\le j\le n$.

\begin{description}{\itemsep=10pt \itemindent=0em \leftmargin=0em}
\item[{[right:]}]
The following are equivalent
\begin{enumerate}
\item\label{matrixcoeffiv}
$\left[f(C)\right]_{i\,j}\in I$ for all $C\in\TnS$ 
\item\label{scalarcoeffiv}
$\left[f_{ih}(C)\right]_{h\,j}\in I$ for all $C\in\TnS$, for all $h\in [i,j]$.
\item\label{pathpolyiv}
$\left\langle f_{ih}, p_{hj}\right\rangle\in\IntR(S^{\ast},I)$
for all $h\in [i,j]$.
\end{enumerate}

\item[{[left:]}]
The following are equivalent
\begin{enumerate}
\item\label{l-matrixcoeffiv}
$\left[\fCl\right]_{i\,j}\in I$ for all $C\in\TnS$ 
\item\label{l-scalarcoeffiv}
$\left[f_{hj}(C)\right]_{i\,h}\in I$ for all $C\in\TnS$, for all $h\in [i,j]$.
\item\label{l-pathpolyiv}
$\left\langle f_{hj}, p_{ih}\right\rangle\in\IntR(S^{\ast},I)$
for all $h\in [i,j]$.
\end{enumerate}
\end{description}
\goodbreak

\begin{proof}
For \textbf{right} substitution:
($\ref{scalarcoeffiv}\Rightarrow \ref{matrixcoeffiv}$) follows
directly from 
\[
\left[f(C)\right]_{i\,j} = 
\sum_{h\in [i, j]} \left[f_{ih}(C)\right]_{h\,j},
\]
which is Lemma~\ref{summationchange}.

($\ref{matrixcoeffiv}\Rightarrow \ref{scalarcoeffiv}$)
Induction on $h$ from $j$ down to $i$. Given 
$m\in [i,j]$, we show (\ref{scalarcoeffiv}) for $h=m$, assuming that 
the statement holds for all values $h\in [i,j]$ with $h>m$.
In the above formula from Lemma~\ref{summationchange}, we let
$C$ vary through $\TnSres{m}{j}$ (as in \ref{TnSres}).
For such a $C$, the summands
$\left[f_{ih}(C)\right]_{h\,j}$ with $h<m$ are zero, by
Remark~\ref{pathpolyproperties} (\ref{pathpolyzero}). 
The summands with $h>m$ are in $I$, by induction hypothesis. 
Therefore 
$\left[f(C)\right]_{i\,j}\in I$ for all $C\in\TnS$ implies 
$\left[f_{im}(C)\right]_{m\,j}\in I$ for all $C\in\TnSres{m}{j}$.
Since $\left[f_{im}(C)\right]_{m\,j}=\left[f_{im}(C^{[m,j]})\right]_{m\,j}$,
by Remark~\ref{resfCijh}, the statement follows for all $C\in\TnS$.

Finally, ($\ref{pathpolyiv}\Leftrightarrow \ref{scalarcoeffiv}$) 
because, firstly, $\left[f_{ih}(C)\right]_{h\,j}=
\left\langle f_{ih}, p_{hj}\right\rangle(C)$, for all $C\in \TnR$, 
and secondly, as far as the image of $\left\langle f_{ih},
p_{hj}\right\rangle$ is concerned, all possible values under
substitution of elements from $S$ for the variables are obtained as
$C$ varies through $\TnS$, because no variables other that $x_{ij}$
with $1\le i\le j\le n$ occur in this polynomial.

For \textbf{left} substitution:
($\ref{l-scalarcoeffiv}\Rightarrow\ref{l-matrixcoeffiv}$)
follows directly from
\[
\left[\fCl\right]_{i\,j} = 
\sum_{h\in [i, j]}  \left[f_{hj}(C)\right]_{i\,h},
\]
which is Lemma~\ref{summationchange}.

($\ref{l-matrixcoeffiv}\Rightarrow \ref{l-scalarcoeffiv}$)
Induction on $h$, from $h=i$ to $h=j$. Given $m\in [i,j]$,
we show $(\ref{l-scalarcoeffiv})$ for $h=m$ under the hypothesis
that it holds for all $h\in [i,j]$ with $h<m$.
In the above formula from Lemma~\ref{summationchange}, the summands
corresponding to $h<m$ are in $I$ by induction hypothesis. If we
let $C$ vary through $\TnSres{i}{m}$ (as in \ref{TnSres}), then, 
for such $C$, the summands 
$\left[f_{hj}(C)\right]_{i\,h}=\langle f_{hj},p_{ih}\rangle(C)$
corresponding to $h>m$ are zero, 
by Remark~\ref{pathpolyproperties} (\ref{pathpolyzero}).
Therefore
$\left[f_{mj}(C^{[i,m]})\right]_{i\,m}\in I$ for all $C\in \TnS$.
But $\left[f_{mj}(C^{[i,m]})\right]_{i\,m}=
\left[f_{mj}(C)\right]_{i\,m}$ for all $C\in \TnS$,
by Remark~\ref{resfCijh}.
Thus, (\ref{l-scalarcoeffiv}) follows by induction.

Finally, ($\ref{l-pathpolyiv}\Leftrightarrow \ref{l-scalarcoeffiv}$)
because, firstly, $\left[f_{hj}(C)\right]_{i\,h}=
\left\langle f_{hj}, p_{ih}\right\rangle(C)$, for all $C\in \TnR$,
and secondly, as far as the image of 
$\left\langle f_{hj}, p_{ih}\right\rangle$ is concerned, all possible
values under substitution of elements of $S$ for the variables are
realized as $C$ varies through $\TnS$, since no variables other than
$x_{ij}$ with $1\le i\le j\le n$ occur in this polynomial.
\end{proof}
\end{lem}

\section{Results for polynomials with matrix coefficients}%
\label{matrixcoeffresultsect}

Considering how matrix multiplication works, it is no surprise that,
for a fixed row-index $i$ and $f\in (\TnR)[x]$, whether the entries of 
the $i$-th row of $f(C)$ are in $I$ for every $C\in\TnS$, depends
only on the $i$-th rows of the coefficients of $f$. Indeed, if
$f=\sum_k F_k x^k$, and $\left[B\right]_i$ denotes the $i$-th row of
a matrix $B$, then 
\[
\left[f(C)\right]_i = \sum_k [F_k]_i C^k.
\]

Likewise, for a fixed column index $j$, whether the entries of
the $j$-th column of $\fCl$ are in $I$ for every $C\in\TnS$
depends only on the $j$-th columns of the coefficients of 
$f=\sum_k F_k x^k$: if $\left[B\right]^j$ denotes the
$j$-th column of a matrix $B$, then
\[
\left[\fCl\right]^j = \sum_k C^k [F_k]^j.
\]

Now we can formulate the precise criterion which the $i$-th rows, or
$j$-th columns, of the coefficients of $f\in (\TnR)[x]$ have to satisfy 
to guarantee that the entries of the $i$-th row of $f(C)$, or the $j$-th 
column of $\fCl$, respectively, are in $I$ for every $C\in \TnS$.

\begin{lem}\label{each-row}
Let $f\in (\TnR)[x]$.
We use the notation of \ref{matrixcoeffnotation}
and Definitions \ref{pathpolydef}.
\ref{matrixsubstitution}, and \ref{scalarproddef}.
\begin{description}
\item[{[right:]}]
Let $1\le i\le n$. The following are equivalent
\begin{enumerate}
\item\label{row-i-inI}
For every $C\in \TnS$, all entries of the $i$-th row of $f(C)$ are in $I$.
\item\label{row-i-each-entry}
For all $C\in\TnS$, for all $h,j$ with $i\le h\le j\le n$, 
$\left[f_{ih}(C)\right]_{h\,j}\in I$.
\item\label{rowi-pathpoly}
For all $h,j$ with $i\le h\le j\le n$,
$\left\langle f_{ih}, p_{hj}\right\rangle\in\IntR(S^{\ast},I)$.
\item\label{fihcriterion}
$f_{ih}\in \TMint{n-h+1}$ for $h=i,\ldots, n$.
\end{enumerate}
\item[{[left:]}]
Let $1\le j\le n$. The following are equivalent
\begin{enumerate}
\item\label{col-j-inI}
For every $C\in \TnS$, all entries of the $j$-th column of $\fCl$ are in
$I$.
\item\label{col-j-each-entry}
For all $C\in \TnS$, for all $i,h$ with $1\le i\le h\le j$,
$\left[f_{hj}(C)\right]_{i\,h}\in I$.
\item\label{colj-pathpoly}
For all $i,h$ with $1\le i\le h\le j$,
$\left\langle f_{hj}, p_{ih}\right\rangle\in\IntR(S^{\ast},I)$.
\item\label{fhjcriterion}
$f_{hj}\in \TMint{h}$ for $h=1,\ldots, j$.
\end{enumerate}
\end{description}
\end{lem}

\begin{proof}
For \textbf{right} substitution: by Lemma~\ref{eachij}, 
(\ref{row-i-inI}-\ref{rowi-pathpoly}) are equivalent
conditions for the $(i,j)$-th entry of $f(C)$ to be in $I$
for all $j$ with $i\le j\le n$, for every $C\in \TnS$.

(\ref{rowi-pathpoly}$\Rightarrow$\ref{fihcriterion})
For each $h$ with $i\le h\le n$, letting $j$ vary from $h$ to $n$
shows that criterion (\ref{eqiv-pathpoly}) of
Proposition~\ref{scalar-coeff-iv-criterion} for
$f_{ih}\in \TMint{n-h+1}$ is satisfied.

(\ref{fihcriterion}$\Rightarrow$\ref{rowi-pathpoly})
For each fixed $h$, $f_{ih}$ satisfies, by criterion (\ref{each-entry}) 
of Proposition~\ref{scalar-coeff-iv-criterion}, in particular, 
$\langle f_{ih}, p_{1k}\rangle\in \Int(S^{\ast},I)$ 
for all $1\le k\le n-h+1$. By Lemma~\ref{image-equiv}, 
$\langle f_{ih}, p_{1k}\rangle(S^{\ast})$ equals
$\langle f_{ih}, p_{h\, h+k-1}\rangle(S^{\ast})$. 
Therefore $\langle f_{ih}, p_{h j}\rangle\in \Int(S^{\ast},I)$
for all $j$ with $h\le j\le n$.

For \textbf{left} substitution: by Lemma~\ref{eachij}, 
(\ref{col-j-inI}-\ref{colj-pathpoly}) are equivalent conditions
for the $(i,j)$-th entry of $\fCl$ to be in $I$ for
all $i$ with $1\le i\le j$, for every $C\in \TnS$.

(\ref{rowi-pathpoly}$\Rightarrow$\ref{fihcriterion})
For each $h$ with $1\le h\le j$, letting $i$ range from $1$ to $h$
shows that criterion (\ref{eqiv-pathpoly}) of
Proposition~\ref{scalar-coeff-iv-criterion} for
$f_{hj}\in \TMint{h}$ is satisfied.

(\ref{fihcriterion}$\Rightarrow$\ref{rowi-pathpoly})
For each $h$ with $1\le h\le j$, 
applying criterion (\ref{each-entry}) of 
Proposition~\ref{scalar-coeff-iv-criterion} to $f_{hj}$
shows, in particular, $\langle f_{hj}, p_{ih}\rangle\in \Int(S^{\ast},I)$
for all $1\le i\le h$.
\end{proof}

We are now ready to prove the promised characterization of
polynomials with matrix coefficients in terms of polynomials
with scalar coefficients:

\begin{thm}\label{matrix-scalar-connection}
Let $f\in (T_n(R))[x]$. Let $F_k=(f_{ij}^{(k)})_{1\le i,j\le n}$ denote 
the coefficient of $x^k$ in $f$ and set 
\[f_{ij}=\sum_k f_{ij}^{(k)} x^k.\]
Then
\[f\in\TMMintn\quad\Leftrightarrow\quad
\forall i,j\;\;  f_{ij}\in \TMint{n-j+1}\]
and
\[f\in\TMMintnl\quad\Leftrightarrow\quad
\forall i,j\;\; f_{ij}\in \TMint{i}.\]
\end{thm}

\begin{proof}
The criterion for $f\in\TMMintn$ is just Lemma~\ref{each-row}
applied to each row-index $i$ of the coefficients of $f\in (T_n(R))[x]$,
and the criterion for $f\in\TMMintnl$ is Lemma~\ref{each-row}
applied to each column index $j$ of the coefficients of $f$.
\end{proof}

With the above proof we have also shown Remark~\ref{uppertriangular} 
from the introduction, which is the representation in matrix form of 
Theorem~\ref{matrix-scalar-connection}.

\section{Applications to null-polynomials and integer-valued polynomials}%
\label{applsect}

\comment{%
\begin{thm} Let $R$ be a commutative ring. The set of
right null-polynomials in $(\TnR)[x]$, that is,
\[
N_{(\TnR)[x]}((\TnR)=\{f\in (\TnR)[x]\mid \forall C\in (\TnR)\>
f(C)=0\}
\]
and the set of left null-polynomials in $(\TnR)[x]$, that is,
\[
N^{\ell}_{(\TnR)[x]}((\TnR)=\{f\in (\TnR)[x]\mid \forall C\in (\TnR)\>
\fCl=0\},
\]
are ideals of $(\TnR)[x]$.
\end{thm}
}

Applying our findings to null-polynomials on upper triangular
matrices, we can describe polynomials with coefficients in $\TnR$ that 
induce the zero-function on $\TnR$ 
in terms of polynomials with coefficients in $R$ that induce the 
zero function on $\TnR$.
As before, we denote substitution for the variable in a polynomial
$f=\sum_k b_k x^k$ in $\TnR[x]$, to the
right or to the left of the coefficients, as
\[f(C)=\sum_k b_k C^k \qquad\textrm{and}\qquad \fCl=\sum_k C^k b_k \]
and define
\begin{align*}
N_{\TnR}(\TnR)&=\{f\in (\TnR)[x]\mid \forall C\in \TnR\;
f(C)=0\}\\
N^{\ell}_{\TnR}(\TnR)&=\{f\in \TnR[x]\mid \forall C\in \TnR\;
\fCl=0\}\\
\NR{k}&=\{f\in R[x]\mid \forall C\in \TnR\; f(C)=0\}
\end{align*}

Null-polynomials on matrix algebras
occur naturally in two circumstances: in connection with null-ideals
of matrices \cite{Fri04icm,Ri16nim}, and in connection with 
integer-valued polynomials.
For instance, if $D$ is a domain with quotient field $K$ and
$f\in(\TnK)[x]$, we may represent $f$ as $g/d$ with $d\in D$ and
$g\in(\TnD)[x]$. Then $f$ is (right) integer-valued on $\TnD$ if
and only if the residue class of $g$ in $T_n(D/dD)[x]$ is a 
(right) null-polynomial on $T_n(D/dD)$ \cite{Fri05pspa}.

From Theorem~\ref{matrix-scalar-connection}
we derive the following corollary (see also Remark~\ref{uppertriangular}):

\begin{cor}\label{nullpolymatrixrep}
If we identify polynomials in $(\TnR)[x]$ with their images in
$T_n(R[x])$ under the isomorphism of \ref{isomorphism}, then\\
$N_{\TnR}(\TnR)=$
\[
\begin{pmatrix}
\NR{n}&\NR{n-1}&\ldots&\NR{2}&\NR{1}\cr
0&\NR{n-1}&\ldots&\NR{2}&\NR{1}\cr
&&\ddots&&\cr
0&0&\ldots&\NR{2}&\NR{1}\cr
0&0&\ldots&0&\NR{1}\cr
\end{pmatrix}
\] 

and $N^{\ell}_{\TnR}(\TnR)=$ 
\[
\begin{pmatrix}
\NR{1}&\NR{1}&\ldots&\NR{1}&\NR{1}\cr
0&\NR{2}&\ldots&\NR{2}&\NR{2}\cr
&&\ddots&&\cr
0&0&\ldots&\NR{n-1}&\NR{n-1}\cr
0&0&\ldots&0&\NR{n}\cr
\end{pmatrix}
\]
\end{cor}

This allows us to conclude:

\begin{thm}\label{ideal}
Let $R$ be a commutative ring. 
The set $N_{\TnR}(\TnR)$ of right null-polynomials on $\TnR$ with
coefficients in $(\TnR)[x]$,
and the set $N^{\ell}_{\TnR}(\TnR)$ of left null-polynomials on
$\TnR$ with coefficients in $\TnR$,
are ideals of $(\TnR)[x]$.
\end{thm}

\begin{proof}
Note that $\NR{m}\subseteq \NR{k}$ for $m\ge k$. Also,
$\NR{m} R[x]\subseteq \NR{m}$ and $R[x]\NR{m} \subseteq \NR{m}$ 
by substitution homomorphism for polynomials with coefficients 
in the commutative ring $R$.
This observation together with matrix multiplication
shows that the image of $N_{\TnR}((\TnR)$ in $T_n(R[x])$
under the ring isomorphism from $(\TnR)[x]$ to $T_n(R[x])$ is 
an ideal of $T_n(R[x])$, and likewise the image of 
$N^{\ell}_{\TnR}(\TnR)$ in $T_n(R[x])$. 
\end{proof}

Now let $D$ be a domain with quotient field $K$.
Applying Theorem~\ref{matrix-scalar-connection}
to integer-valued polynomials (see also Remark~\ref{uppertriangular})
\begin{align*}
\TMMintnD&=\{f\in (\TnK)[x]\mid \forall C\in\TnD\; f(C)\in\TnD\}\\
\TMMintnDl&=\{f\in (\TnK)[x]\mid \forall C\in\TnD\; \fCl\in\TnD\}\\
\TMintD{n}&=\{f\in K[x]\mid \forall C\in\TnD\; f(C)\in\TnD\}
\end{align*}
yields the following corollary: 

\begin{cor}
If we identify polynomials in $(\TnK)[x]$ with their images in
$T_n(K[x])$ under the isomorphism of \ref{isomorphism}, then\\

$\TMMintnD=$\\
\[
\begin{pmatrix}
\TMintD{n}&\TMintD{n-1}&\ldots&\TMintD{2}&\TMintD{1}\cr
0&\TMintD{n-1}&\ldots&\TMintD{2}&\TMintD{1}\cr
&&\ddots&&\cr
0&0&\ldots&\TMintD{2}&\TMintD{1}\cr
0&0&\ldots&0&\TMintD{1}\cr
\end{pmatrix}
\]

and\ \ $\TMMintnDl=$\\
\[
\begin{pmatrix}
\TMintD{1}&\TMintD{1}&\ldots&\TMintD{1}&\TMintD{1}\cr
0&\TMintD{2}&\ldots&\TMintD{2}&\TMintD{2}\cr
&&\ddots&&\cr
0&0&\ldots&\TMintD{n-1}&\TMintD{n-1}\cr
0&0&\ldots&0&\TMintD{n}\cr
\end{pmatrix}
\]
\end{cor}

We can see the following:

\begin{thm}\label{ring}
Let $D$ be a domain with quotient field $K$. 
The set $\TMMintnD$ of right integer-valued polynomials with 
coefficients in $(\TnK)[x]$ and the set $\TMMintnDl$ of
left integer-valued polynomials with coefficients in $(\TnK)[x]$
are subrings of $(\TnK)[x]$.
\end{thm}

\begin{proof}
Note that $\TMintD{m}\subseteq \TMintD{k}$ for $m\ge k$. Together with
substitution homomorphism for polynomials with coefficients in $K$, 
this implies
\[
\TMintD{i}\cdot \TMintD{j}\subseteq \TMintD{\min(i,j)}.
\]
This observation together with matrix multiplication shows that the 
image of $\TMMintnD$ in $T_n(K[x])$ under the ring isomorphism of
$(\TnK)[x]$ and $T_n(K[x])$ is a subring of $T_n(K[x])$, and, likewise,
the image of $\TMMintnDl$ is a subring of $T_n(K[x])$.
\end{proof}

\bibliography{pfnc}
\bibliographystyle{siamese}

\end{document}